%% file: size-condition-oriented-diameter-Arxiv.tex
\documentclass{amsart}

\usepackage{tkz-graph, amssymb, enumerate}

\usetikzlibrary{arrows,decorations.pathmorphing,backgrounds,positioning,fit,petri,calc}

\title{A Size Condition for Diameter Two Orientable Graphs}

\newtheorem{theorem}{Theorem}[section]
\newtheorem{lemma}{Lemma}[section]

\newtheorem{corollary}{Corollary}[section]
\newtheorem{definition}{Definition}[section]

\author[Cochran]{Garner Cochran}
\address{Garner Cochran\\Department of Mathematics and Computer Science\\ Berry College\\2277 Martha Berry Hwy NW\\  Mt Berry GA 30149\\ USA}
\email{gcochran@berry.edu}
\author[Czabarka]{\'Eva Czabarka}
\address{\'Eva Czabarka\\Department of Mathematics \\ University of South Carolina \\ Columbia SC 29212 \\ USA
\and Visiting Professor\\ Department of Pure and Applied Mathematics\\ University of Johannesburg\\South Africa}
\email{czabarka@math.sc.edu}
\author[Dankelmann]{Peter Dankelmann}
\address{Peter Dankelmann\\ Department of Pure and Applied Mathematics\\ University of Johannesburg\\South Africa}
\email{pdankelmann@uj.ac.za}
\author[Sz\'ekely]{L\'aszl\'o Sz\'ekely}
\address{L\'aszl\'o Sz\'ekely\\Department of Mathematics \\ University of South Carolina \\ Columbia SC 29212 \\ USA
\and Visiting Professor\\ Department of Pure and Applied Mathematics\\ University of Johannesburg\\South Africa}
\email{szekely@math.sc.edu}
\thanks{The first author was supported by a SPARC Graduate Research Grant from the Office of the Vice President for Research at the University of South Carolina. 
The third author was supported in part by the National Research Foundation of South Africa, grant number 103553. 
The fourth author was supported in part by the NSF DMS, grant number 1600811.}

\begin{document}

\maketitle

\begin{abstract}
It was conjectured by Koh and Tay 
[Graphs Combin.\ 18(4) (2002), 745--756] 
that for $n\geq 5$ every simple graph of order $n$ and size at
least $\binom{n}{2}-n+5$ has an orientation of diameter two.
We prove this conjecture and hence determine for every
$n\geq 5$ the minimum value of $m$ such that every
graph of order $n$ and size $m$ has an orientation
of diameter two.

\end{abstract}

\section{Introduction}

This paper is concerned with the diameter of orientations of graphs. 
Given a graph $G$, an {\em orientation} $O_G$ of $G$ is a 
digraph obtained from $G$ by assigning a direction
to every edge of $G$. 
The {\em distance} between two vertices $u$ and $v$ in a graph or digraph $H$, 
denoted by $d_H(u,v)$, is the 
minimum length of a $(u,v)$-path in $H$; it is infinite if there is no
such path. The {\em diameter} of $H$ is the largest of the distances
between all pairs of vertices, it is denoted by ${\rm diam}(H)$. The
well-known Robbin's Theorem \cite{Rob1939} states that a connected graph has an orientation of
finite diameter if and only if it is bridgeless.
The {\em oriented diameter} of a graph is the minimum diameter of an 
orientation of $G$. 
 Chv\'{a}tal and Thomassen
\cite{ChvTho1978} showed that there is a function $f$ such that every 
bridgeless graph of diameter $d$ has an orientation of diameter at most $f(d)$. 
The determination of the exact values of this function appears 
extremely difficult. Chv\'{a}tal and Thomassen
\cite{ChvTho1978} showed that every bridgeless graph of diameter two
has an orientation of diameter at most six, and that this value is attained
by the Petersen graph, so $f(2)=6$. Already the value $f(3)$ is not known. 
Egawa and Iida \cite{EgaIid2007} and, independently, Kwok, Liu and West 
\cite{KwoLiuWes2010} showed that the oriented diameter of a bridgeless graph of diameter three
is at most $11$. 
In \cite{KwoLiuWes2010} an example of a graph of diameter $3$ and oriented
diameter $9$ was given. Hence $9 \leq f(3) \leq 11$. 
It was shown by Bau and Dankelmann  \cite{BauDan2015} that every bridgeless
graph of order $n$ and minimum degree $\delta$ has an orientation of diameter
at most $\frac{11n}{\delta+1} +O(1)$. Surmacs \cite{Sur2017} improved this bound
to $\frac{7n}{\delta+1} + O(1)$. An upper bound on the oriented diameter
terms of maximum degree was given by Dankelmann, Guo and Surmacs \cite{DanGuoSur2018}. 

Chv\'{a}tal and Thomassen
\cite{ChvTho1978} further showed that the problem of deciding whether a given graph
has an orientation of diameter two is NP-complete. Even for complete
multipartite graphs the problem which such graphs have an orientation
of diameter two has not been solved completely, see \cite{KohTan1996,KohTan19962}. 
Hence it is desirable to
have sufficient conditions that guarantee that a given graph has an 
orientation of diameter two.

In this paper we relate the existence of an orientation of diameter two
of a graph of given order to its size. 
F\"{u}redi, Hor\'ak, Pareek and Zhu \cite{FurHorParZhu1998} gave an
asymptotically sharp lower bound on the number of edges in a graph
of given order 
that admits an orientation of diameter two. The purpose of this paper 
is to determine for every $n\geq 5$ the
minimum value $m(n)$ such that every simple graph of order $n$ and size 
at least $m(n)$ has an orientation of diameter two. 

For $n\geq 5$, the graph $G_n$, obtained from a complete graph on $n-1$ 
vertices by adding a new vertex $v$ and edges joining $v$ to three 
vertices in the complete graph, does
not have an orientation of diameter two. 
Indeed, suppose to the contrary that $G_n$ has an orientation $O_n$ of 
diameter two. Then $v$ has either two in-neighbors and one out-neighbor, 
or vice versa. We may assume the former.
Let $u$ be the out-neighbor and $y_1,y_2$ be the two in-neighbors of $v$ in $O_n$. 
Since every vertex is at distance at most two from $v$ in $O_{n}$, for every vertex $w\in V(G_n)-\{u, v\}$ 
the edge $uw$ is oriented from $u$ to $w$. Hence, if $x \in V(G_n)-\{u,v,y_1,y_2\}$ 
any $(x,u)$-path in $O_n$ goes through $v$ and has thus
length at least three, a contradiction to $O_n$ having diameter two.  
Hence $G_n$ has no orientation of diameter two. It follows that
$m(n) \geq m(G_n) + 1 = \binom{n}{2}-n+5$ for $n\geq 5$. This was observed by
Koh and Tay \cite{KohTay2002}, who conjectured that 
this construction is best possible, and so 
$m(n) = \binom{n}{2}-n+5$ for $n\geq 5$. 
It is the aim of this paper to show that this conjecture is true
by proving the following theorem.

\begin{theorem} \label{theo:main-result}
Let $G$ be a simple graph of order $n$, where $n\geq 5$, and size 
at least $\binom{n}{2}-n+5$. Then $G$ has an orientation of diameter two.  
\end{theorem}

Our proof of Theorem \ref{theo:main-result} consists of a sequence of lemmata.
An outline of the proof is as follows. 
We suppose to the contrary  that the theorem is false and that $G$ is a counterexample of minimum 
order, and among those, minimum size. Our proof focuses on the complement $\overline{G}$ of $G$, defined 
as the graph on the same vertex set as $G$, where two vertices are adjacent in 
$\overline{G}$ if and only if they are not adjacent in $G$. 

In Section \ref{section:sufficient conditions for an orientation of diameter two} we give 
some sufficient conditions for graphs to have an orientation
of diameter two, and we present several graphs that have an orientation
of diameter two. 
In Section \ref{section:some properties of B} we present some properties of
the graph $\overline{G}$ that will be useful later; in particular we show that 
each component of $\overline{G}$ contains neither three independent vertices nor
two non-adjacent vertices that share more than one neighbour.  These results,
together with some results in Section \ref{sec:On Tree Components of $B$} on the components 
of $\overline{G}$ that are trees, will be used in Section \ref{section:describing the components of B}
to show that the components of $\overline{G}$ are short paths, and possibly an additional component that is one of four 
types of graphs on at most $6$ vertices. 
In Section~\ref{sec:main-result} we complete the proof by showing that the presence of any of these four types of graphs either
allows us to apply certain reductions to the graph $G$ to obtain a smaller 
counterexample $G^{\prime}$, or that $G$ is one of the graphs in the list of graphs with 
an orientation of diameter two presented in 
Section \ref{section:sufficient conditions for an orientation of diameter two}, 
so $G$ is not a counterexample.  Finally, we conclude the proof by dealing with the case that all components of $\overline{G}$ are trees.

\section{Notation}
All graphs and digraphs in this paper have neither loops nor multiple edges, i.e. they are unoriented or oriented simple graphs. 
Let $G$ be a graph of order $n=n(G)$ and size $m=m(G)$.
We define $G_1=(V_1,E_1)$ to be a subgraph of $G_2=(V_2,E_2)$ when $V_1\subseteq V_2$ and $E_1\subseteq E_2$. We denote this as $G_1\trianglelefteq G_2$. 
We define the {\em excess} of  $G$
by ${\rm ex}(G)= m(G) - n(G)$. 
We find it convenient to consider $G$ and $\overline{G}$ as 
obtained by colouring the edges of a complete graph on $n$ vertices either red or blue, 
with the edges of $G$ being the red, and the edges of $\overline{G}$ 
as blue edges. Accordingly, we usually denote
$G$ as $R$, and $\overline{G}$ as $B$. We denote the vertex set common 
to $R$ and $B$ by $V$.  If $W\subseteq V$,  then 
the red and blue subgraph induced by $W$ in $R$ and $B$, respectively,
is denoted by $R[W]$ and $B[W]$.  

Let $u, v$ be vertices  of a graph $G$ or digraph $O_G$. If $uv \in E(G)$ then 
we say that $u$ and $v$ are adjacent in $G$ and that $u$ is a neighbor of $v$.
The set of all neighbors of $v$ is the neighborhood of $v$ in $G$, denoted by
$N_G(v)$. The closed neighborhood $N_G[v]$ of $v$ in $G$ is defined as $N_G(v) \cup \{v\}$.   
If $\overrightarrow{uv}$ is a directed edge of $O_G$, then we say that $v$ is an 
out-neighbor of $u$ and that $u$ is an in-neighbor of $v$.  
The {\em degree} of vertex $v$ in $G$ is the number of neighbors of $v$, 
it is denoted by ${\rm deg}_G(v)$.

By $K_n$, $P_n$, $C_n$, and $K_{a,b}$ we mean the complete graph on $n$ vertices, 
the path on $n$ vertices, the cycle on $n$ vertices, and the complete bipartite graph whose partite
sets have $a$ and $b$ vertices, respectively. If $G$ and $H$ are graphs, then $G\cup H$ is the 
disjoint union of $G$ and $H$. If $a$ is a positive integer, then 
$aG$ is the disjoint union of $a$ copies of $G$, so the edgeless graph on $n$ 
vertices is denoted by $nK_1$. 

If $U$ and $W$ are disjoint subsets of $V$ then $U \rightarrow W$
indicates that for all $x\in U$ and $y \in W$ that are adjacent in $R$  we 
orient the edge $xy$ as $\overrightarrow{xy}$, i.e., from $x$ to $y$. 
We write $u \rightarrow W$ instead of $\{u\}\rightarrow W$, 
and similarly $U \rightarrow w$ and $u \rightarrow w$ instead of
$U\rightarrow\{w\}$ and $\{u\}\rightarrow\{w\}$. 

If $A, B$ are sets of vertices in $H$, then their distance, $d_H(A,B)$, is defined as the Hausdorff distance
$\min_{u\in A, v\in B} d_H(u,v)$. $d_H(u,B)$ and $d_H(A,v)$ are
defined analogously. 

As usual, $[n]=\{1,2,3,\ldots,n\}$ and for a set $A$ and 
$k\in\mathbb{N}$, $\binom{A}{k}$ is the collection of $k$-element subsets of $A$.

\begin{definition} 
Let $k,\ell\in\mathbb{Z}^+$. A $(k,\ell)$-{\em dumbbell}, denoted by $D_{k, \ell}$, 
is a graph of order $k + \ell$ obtained from the disjoint union of two complete graphs 
$K_k$ and $K_{\ell}$ by adding an edge joining a vertex of $K_k$ to
a vertex of $K_{\ell}$. A {\em short} $(k,\ell)$-{\em dumbbell}, denoted by
$S_{k, \ell}$, is a graph of order 
$k+\ell-1$ obtained from the disjoint union of two complete graphs 
$K_k$ and $K_{\ell}$ by identifying a vertex of $K_k$ and 
a vertex of $K_{\ell}$.  \\
A $(k,\ell)$-dumbbell is {\em proper} if it not a tree, i.e., if $\max(k,\ell)\ge 3$.
A short $(k,\ell)$-dumbbell is {\em proper} if it is neither complete, nor a tree,
nor a dumbbell, i.e., if $\min(k,\ell)\ge 3$.
\end{definition}

Note that a $(k,\ell)$-dumbbell is a tree if and only if $\max(k,\ell)\le 2$, 
The dumbbells that are trees are paths $P_i$ on $2 \leq i\le 4$ vertices. 
A short $(k,\ell)$-dumbbell is a dumbbell or a complete graph if and only
if $\min(k,\ell)\le 2$.

\section{Sufficient conditions for a diameter two orientation} 
\label{section:sufficient conditions for an orientation of diameter two}

In this section we present a few sufficient conditions for the existence of a diameter two orientation of a graph. 
Using these conditions we obtain a list of several graphs that have diameter two orientations. This list will be used extensively in later sections. 

\begin{definition} 
Let $W \subseteq V$. An orientation $O_W$ of $R[W]$ is {\em good} if
there exists a partition of $W$ into two sets $U_1$ and $V_1$, which 
we call the {\em partition classes} of $W$ (or of $O_W$), such that \\
{\rm (i)} $d_{O_W}(x,y)\leq 2$ whenever $x$ and $y$ are both in $U_1$ or both in $V_1$. \\
If in addition  \\
{\rm (ii)} every vertex in $U_1$  has an in-neighbor and an out-neighbor 
in $V_1$ and vice versa,  \\
then $O_W$ is a {\em non-trivial good orientation}. If $R[W]$ has a (non-trivial) good orientation,
then we sometimes say simply that
$W$ has a (non-trivial) good orientation.
\end{definition}

The following lemma is based on a construction of digraphs of diameter
two with no $2$-cycles having close to the minimum number or edges
by  F\"{u}redi, Hor\'{a}k, Pareek and Zhu \cite{FurHorParZhu1998}.

\begin{lemma}\label{lem:Kab} 
Let $a,b \in \mathbb{N}$ with $2\le a\le b\le\binom{a}{\lfloor a/2\rfloor}$. 
If $R[W]$ contains $K_{a,b}$ as a spanning subgraph, then $R[W]$ has a non-trivial good orientation. 
If $R[W]$ is isomorphic to $K_{1,1}$, then $R[W]$ has a good orientation. 
\end{lemma}

\begin{proof} Any orientation of $K_{1,1}$ is vacuously good, so it suffices to show that $K_{a,b}$ has a non-trivial good orientation
for all $2\le a\le b\le\binom{a}{\lfloor a/2\rfloor}$.

Let the partite classes of $K_{a,b}$ be $U_1=\{x_1,\ldots, x_a\}$ and 
$V_1=\{y_1,\ldots, y_b\}$ and set $c=\lfloor \frac{a}{2}\rfloor-1$. Consider an injection
$f:[b]\rightarrow\binom{[a]}{c+1}$ such that for $i\in[a]\subseteq[b]$ we have $f(i)=\{i,\ldots,i+c\}$, where numbers in $f(i)$ are taken modulo $a$.
Such an injection exists by the conditions on $a$, $b$ and $c$. Orient the edge
$y_i x_j$ as $\overrightarrow{y_ix_j}$ if $j\in f(i)$, and as $\overrightarrow{x_jy_i}$ otherwise. 
For $i\ne k, i,k\in[b]$, both $f(i)\setminus f(k)$ and $f(k)\setminus f(i)$ are nonempty, 
ensuring a directed path of length $2$ in both directions between $y_i$ and $y_k$. 

Now take $i,k$ such that 
$1\le i<k\le a$.
If $k-i\le c$,
 let $\ell\in[a]$ such that $\ell\equiv k+c\mod a$; 
 we have that
$i\in f(i)\setminus f(k)$ an $\ell\in f(k)\setminus f(i)$.
If $k-i> c$, let $\ell=i+c$; we have that $k\in f(k)\setminus f(i)$ and $\ell\in f(i)\setminus f(k)$.
This ensures a directed path of length $2$ 
in both directions between $x_i$ and $x_k$. So $K_{a,b}$ has a good orientation.

As every vertex $y_i\in V_1$ has $\lfloor \frac{a}{2} \rfloor$ in-neighbors 
and $\lceil \frac{a}{2} \rceil$ out-neighbors in $U_1$, it has at least one of each. For each $x_i\in U_1$, the arc $\overrightarrow{y_ix_i}$ exists, and the arc $\overrightarrow{x_iy_{i-1}}$ exists. Hence $K_{a,b}$ has a non-trivial good orientation. 
\end{proof}

\begin{definition} Let $\ell\ge k$ be positive integers. We define $K_\ell\boxplus K_k$ as the disjoint union of $K_\ell$ and $K_k$ together with a set of edges $M^{\star}$ that
match every vertex of $K_k$ to a vertex of $K_{\ell}$. 
\end{definition}

\begin{lemma}\label{lem:Kabsubgraph} 
Let $a,b \in \mathbb{N}$ with $3\le a\le b\le 2a$.
If $R[W]$ contains $K_{a,b}$ as a spanning subgraph with partite sets $X$ and $Y$ such that  $B[Y]\trianglelefteq K_a\boxplus K_{b-a}$,
 then $R[W]$ has a non-trivial good orientation. 
\end{lemma}

\begin{proof}
Let $W=X\cup Y$ where $X=\{x_1,\ldots,x_a\}$, $Y=\{y_1,\ldots,y_{b}\}$. It suffices to prove that $R[W]$ has a non-trivial good orientation
when the edges of $B$ are the union of the edges of the complete graphs on $X$, $\{y_1,\ldots,y_a\}$ and $\{y_{a+1},\ldots,y_{b}\}$ together with the edges
$\{y_iy_{a+i}:i\in[b-a]\}$.

We will provide an appropriate orientation of the red edges.

For $i\in [a]$, orient the edges $x_iy_i$ as $\overrightarrow{x_iy_i}$. For $i,j\in [a]$, where $i\neq j$, orient the edges $x_iy_j$ as $\overrightarrow{y_jx_i}$. 
Note that, as $a>2$, this already ensures that for all $i,j\in[a]$, there is a path of length at most two from $x_i$ to $x_j$ and from $y_i$ to $y_j$, and vertices in
$\{x_1,\ldots,x_1,y_1,\ldots,y_a\}$ have both an in-neighbor and an out-neighbor in $R$.

For $i\in [b-a]$, orient the edges $x_iy_{a+i}$ as $\overrightarrow{y_{a+i}x_i}$. 
For $i,j\in [b-a]$, $i\ne j$, orient the edges $x_iy_{a+j}$ as $\overrightarrow{x_iy_{a+j}}$. This ensures that
for all $i,j\in[b-a]$ and $j\in[a]\setminus\{i\}$ there is an oriented path of length at most two from $y_{a+i}$ to $y_{a+j}$ and from $y_{a+i}$ to $y_i$ (through $x_i$); 
and all vertices of $W$ have an in-neighbor and an out-neighbor in $R$.

For $i\in[a]\setminus[b-a]$ and $j\in[b-a]$, orient the edges $x_iy_{a+j}$ as  $\overrightarrow{x_iy_{a+j}}$. 
This ensures that for all $j\in[b-a]$ and $k\in[a]$ there is an oriented path from $y_k$ to $y_{a+j}$ (through an $x_{\ell}$ where $\ell\in[a]\setminus\{k,j\}$).

Finally, for $i,j\in [b-a]$, with $i\neq j$, orient the edges $y_{a+i}y_j$ as $\overrightarrow{y_{a+i}y_j}$. The resulting
orientation of $R[W]$ is non-trivially good.
\end{proof}

\begin{corollary}\label{cor:paths}
	For a vertex set $W\subseteq V$, if $B[W]$ is a disjoint union of paths and
	the components of $B[W]$ can be partitioned into sets $X$ and $Y$ such that $|X|=a$ and $|Y|=b$ 
	for some $3\leq a\leq b\leq 2a$, then $R[W]$ has a non-trivial good orientation.
\end{corollary}

\begin{proof}
	Let $B[W]$ be the disjoint union of paths which can be partitioned into sets $X$ and $Y$ such that $|X|=a$ and $|Y|=b$ where $3\leq a\leq b\leq 2a$. Then $R[W]$ has $K_{a,b}$ as spanning subgraph with partite sets $X$ and $Y$. Moreover, $Y$ can be partitioned into two sets $Y_a$ and $Y_{b-a}$ of cardinality $a$ and $b-a$ respectively, such that $B[Y]$ contains at most one edge joining a vertex in $Y_a$ to a vertex in $Y_{b-a}$. Hence, $B[Y]\trianglelefteq P_b\trianglelefteq K_a\boxplus K_{b-a}.$
\end{proof}

\begin{lemma}\label{lem:2orient}
Assume that $V$ can be partitioned into two disjoint sets $W$ and $Z$ so that there is no edge in $B$ joining a vertex in $W$ to a vertex in $Z$. Furthermore, assume that $R[W]$ has a non-trivial good orientation, and one of the following holds for $Z$:\\
{\rm (i)} $Z$ has a non-trivial good orientation, or\\
{\rm (ii)} $|Z|= 3$ and the vertices in $Z$ are isolated in $B$, or\\ 
{\rm (iii)} $|Z|=2$, \\ 
then $R$ has an orientation of diameter $2$. 
\end{lemma}

\begin{proof} 
Let $O_W$ be a non-trivial good orientation of $R[W]$ with a 
corresponding partition of $W$ into sets $U_1$ and $V_1$. We will extend it to a non-trivial good orientation of $V$.

Proof of (i): Let $O_Z$ be a non-trivial good orientation of $R[Z]$ with a 
corresponding partition of $Z$ into sets $U_2$ and $V_2$. We assign the orientation $U_1\rightarrow U_2$, $U_2\rightarrow V_1$, $V_1\rightarrow V_2$, and $V_2\rightarrow U_1$. We also include $O_W$ and $O_Z$ in the orientation. It is easy to verify that this in indeed a non-trivial orientation of diameter 2.

Proof of (ii) and (iii): Let $Z=\{y_1, \ldots, y_k\}$ ($k\in\{2,3\}$). 
If $k=3$, orient $R[Z]$ as $y_1\rightarrow y_2\rightarrow y_3\rightarrow y_1$. 
For the remaining red edges,
orient $U_1 \rightarrow y_1$ and $y_1 \rightarrow V_1$, and for $j\in[k]\setminus\{1\}$
orient $y_j \rightarrow U_1$ and 
$V_1 \rightarrow y_j$. Orient any remaining red edges arbitrarily.
It is easy to verify that this is indeed a non-trivial orientation of diameter two. 
\end{proof}

\begin{lemma} \label{la:Garner-list}
The following graphs have an orientation of diameter two:
\begin{enumerate}[{\rm (1)}]
	\item $\overline{Q\cup 7K_1}$,  where $Q\in\{K_4,D_{4,2},D_{4,1}\}$ 
	\item $\overline{D_{4,3} \cup 8K_1}$,
	 \item $\overline{Q \cup 6K_1}$ and  $\overline{Q \cup K_2 \cup 5K_1}$, where $Q\in\{D_{3,3},S_{3,3}\}$
	 \item $\overline{Q\cup aP_1\cup bP_2}$, with $a,b\geq 0$ and $a+b=5$, where $Q\in\{D_{3,2},C_5,D_{3,1},K_3\}$
	 \item\label{la:Gllast} $\overline{aP_1\cup bP_2 \cup cP_3\cup dP_4}$, with $a,b,c,d\geq 0$ and $a+b+c+d=5$.
\end{enumerate}
	 In particular by case~{\rm (\ref{la:Gllast})} {\rm Theorem~\ref{theo:main-result}} holds for $5\le n\le 7$.
\end{lemma}

\begin{proof}
	We either directly give the orientation (for small graphs in case (\ref{la:Gllast})) or 
	find a partition of $V$ into two disjoint sets $W$ and $Z$  for which the conditions of Lemma \ref{lem:2orient} hold. 
	We will do the latter by exhibiting a quadruple $(U_1,V_1,U_2,V_2)$ of subgraphs of $B$  whose vertices partition
	$V$. This signifies that $Z=V(U_1)\cup V(V_1)$, $B[W]=U_2\cup V_2$, all edges between $Z$ and $W$ are red, $R[W]$ has a non-trivial good
	orientation with partition classes $U_2$ and $V_2$, and either $|Z|=2$ (i.e. both $U_1$ and $V_1$ are the singleton $K_1$ and
	$B[Z])\in\{K_2,2K_1\}$), or $|Z|=3$ and the vertices in $Z$ are isolated in $B$, or $R[Z]$ has a 
	non-trivial good orientation with partition classes $U_1$ and $V_1$ (and consequently $B[Z]=U_1\cup V_1)$. 
	
	The proofs of each case in the theorem follow. 
	\begin{enumerate}[(1)]
		\item $B=Q\cup 7K_1$, where $Q\in\{K_4,D_{4,2},D_{4,1}\}$. As $4\le n(Q)\le 6$,
		the quadruple $(K_1,K_1,Q,5K_1)$ gives an orientation of diameter two by Lemmata~\ref{lem:Kab} and~\ref{lem:2orient}.
		
		\item $B=D_{4,3}\cup 8K_1$. We use quadruple $(K_1,K_1,6K_1,D_{4,3})$. Since $6K_1$ and $D_{4,3}$ form a partition of $B$ into 
		two graphs $U_2$ and $V_2$, with $n(U_2)=6$ and $n(V_2)=7$,  Lemma~\ref{lem:Kab} 
		gives that $W$ has a non-trivial good orientation. 
		Since $|Z|=2$, Lemma \ref{lem:2orient} gives a diameter two orientation of $R$.
		
		\item $B\in\{Q \cup 6K_1,Q \cup K_2 \cup 5K_1\}$ where $Q\in\{D_{3,3},S_{3,3}\}$.
		 In both cases quadruple $(K_1,K_1,4K_1,Q)$  gives 
		the required orientation by Lemmata~\ref{lem:Kab} and ~\ref{lem:2orient}.

		\item  $B=Q \cup aP_1\cup bP_2$, with $a,b\geq 0$ and $a+b=5$, where\\ $Q\in\{D_{3,2},C_5,D_{3,1},K_3\}$. Then
		$Q=K_3$ or $n(Q)\in\{4,5\}$. As $\max(a,b)\ge 3$, there are two paths of 
		the same size. Choose a pair of such paths of minimum order $i$ (so $i\in\{1,2\}$), and let $H$ be the union of the remaining three paths. 
		Clearly $3\leq n(H)\leq 6$.\\
		Consider the quadruple $(P_i,P_i,H,Q)$.\\
		If $n(Q)=n(H)=3$ or $n(Q)\ne 3\ne n(H)$, then by Lemmata~\ref{lem:Kab} and ~\ref{lem:2orient} we have the required orientation.\\ 
		If $n(H)=3\ne n(Q)$, notice that $D_{3,2}\trianglelefteq K_3\boxplus K_2$, $C_5\trianglelefteq K_3\boxplus K_2$ and $D_{3,1}= K_3\boxplus K_1$
		and use Lemmata~\ref{lem:Kabsubgraph} and~\ref{lem:2orient} to find an orientation of diameter two. \\
		If $n(H)\ne 3=n(Q)$, then the fact that $H$ is the disjoint union of paths gives that $H \trianglelefteq K_3\boxplus K_{n(H)-3}$. 
		Lemmata~\ref{lem:Kabsubgraph} and~\ref{lem:2orient} give the required orientation.

		\item $B=aP_1\cup bP_2 \cup cP_3\cup dP_4$, with $a,b,c,d\geq 0$ and $a+b+c+d=5$. \\
		All cases where $n(G)<8$ (i.e. when $a+2b+3c+4d\le 7$) and the case where $a=4$, $b=0$, $c=0$, and $d=1$ were done by computer search. See Figure 1 for the orientations of these graphs.

		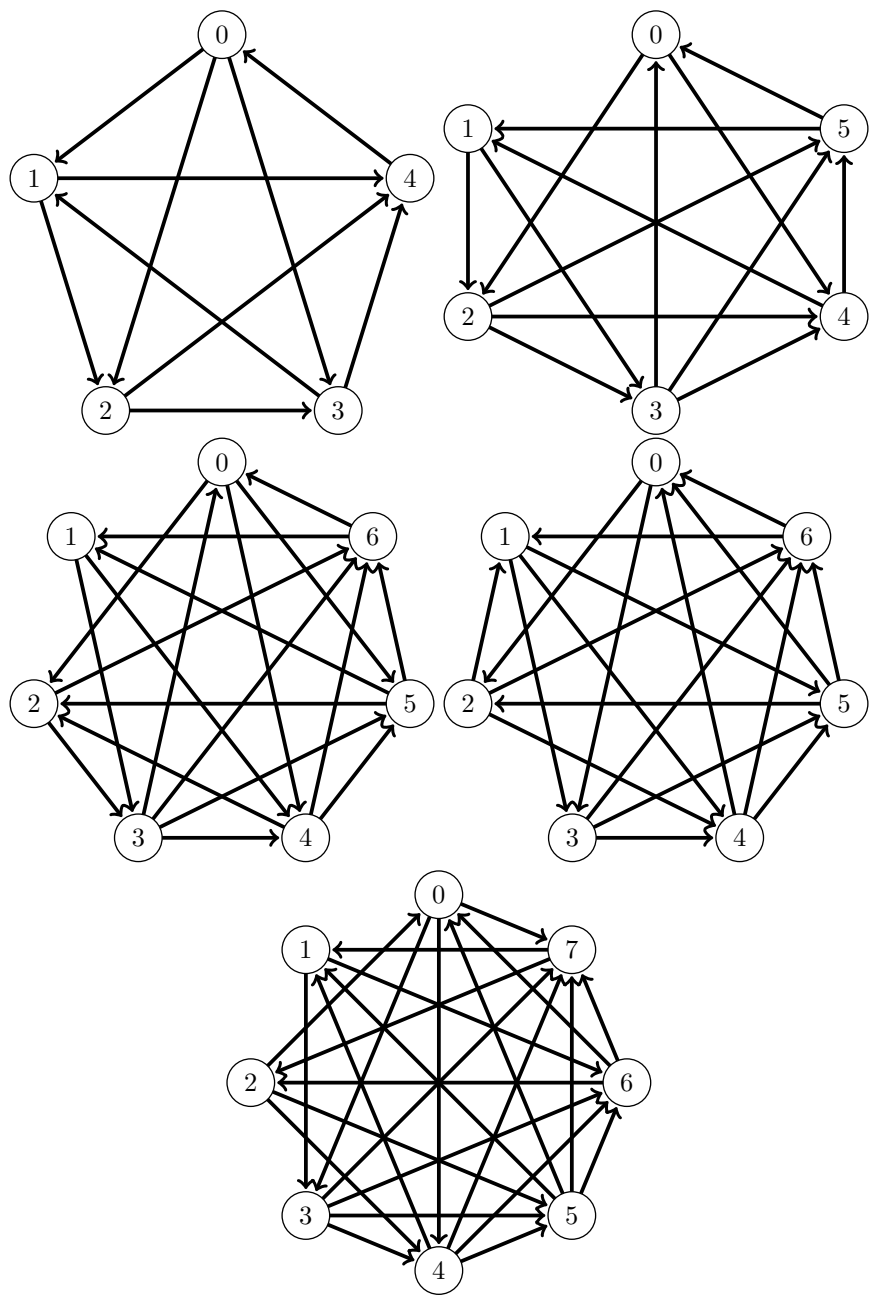
\begin{figure}\label{fig:orientations}
  			\caption{Orientations of graphs where either $n(G)<8$ or $a=4$, $b=0$, $c=0$, and $d=1$.}
 			\centering
 			\vspace{5mm}
			\input{AllCases.tikz}		
		\end{figure}

		For $8\leq n(G)\leq 9$ and we are not in the case $B=P_4\cup 4P_1$, we will consider partitions which use Corollary \ref{cor:paths} and Lemma \ref{lem:2orient}. If $B=P_3\cup P_2\cup 3P_1$, consider the partition $(K_1,K_1,3P_1,P_3)$. If $B=2P_1\cup 3P_2$, consider the partition $(K_1,K_1,P_2\cup P_1,P_2\cup P_1)$. If $B=P_4\cup P_2\cup 3P_1$, consider the partition $(K_1,K_1,3P_1,P_4)$. If $B=2P_3\cup 3P_1$, consider the partition $(P_1,P_1,P_3,P_3\cup P_1)$. If $B=P_3\cup 2P_2\cup 2P_1$, consider the partition $(P_1,P_1,P_3,2P_2)$. If $B=4P_2\cup P_1$, consider the partition $(K_1,K_1,2P_2,2P_2\cup P_1)$. This considers all cases where $n(G)\leq 9$.

		Let $n(G)\geq 10$. As $\max(a,b,c,d)\ge 2$, we again have two paths of the same length. Let $H$ be the union of two paths $P_i$ of the same length where $i$ is chosen
		to be minimum possible, and the remaining three paths be $P_j,P_k,P_\ell$ where without loss of generality $k\le \ell\le j$. We have
		$2i+j+k+\ell=n(G)\ge 10$, so (since $j\ge k\geq\ell$) $\frac{10-2i}{3}\leq j\leq 4$ and $k+\ell\le 2j$. We have two cases.
		\\[1mm]
		{\sc Case 1:} $i=1$\\ 		
		 As $\frac{8}{3}\leq j\leq 4$, we have $j\in\{3,4\}$ and $j\leq 4\leq 10-j-2\leq k+\ell\leq 2j$. Take the quadruple $(P_1,P_1,P_j,P_k\cup P_\ell)$; 
		Lemmata~\ref{lem:Kabsubgraph} and~\ref{lem:2orient} give the required orientation.\\[1mm]
		{\sc Case 2:} $i\geq2$\\ 		
		By the definition of $i$ we must have $\max(k,\ell)>1$, so $k+\ell\ge 3$. 		
		If $j=2$, this gives $i=j=k=\ell=2$ and $G=K_{10}-M$, which has the required orientation by using Corollary \ref{cor:paths} and Lemma \ref{lem:2orient} with the partition $(K_1,K_1,2P_2,2P_2)$, so assume $j\ge 3$. 
		Now either $3\leq k+\ell\le j\leq 4$ or $3\leq j\leq k+\ell\le 2j$ and in both cases the quadruple 
		$(P_i,P_i,P_j,P_k\cup P_\ell)$ with Lemmata~\ref{lem:Kabsubgraph} and~\ref{lem:2orient} give the required orientation.
	\end{enumerate}	
	\end{proof}

\begin{definition}
Let $W\subseteq V$ such that $B[W]$ is the union of one or more components 
of $B$. We say that $W$ is a {\em reducible unit} if  $R[W]$ has a good orientation. 
We say that $W$ is a {\em reduction} if $R[W]$ has a non-trivial good orientation 
and ${\rm ex}(B[W]) \geq -1$.  
\end{definition}

\section{Properties of $B$}
\label{section:some properties of B}

From now on we assume that $G$ is a minimal counterexample, that
is, $G$ is a graph on $n$ vertices, $n\geq 5$, and at least 
$\binom{n}{2} -(n-5)$ edges that has no orientation of diameter two, 
and among those graphs let $G$ be a graph of minimum order and
of minimum size. Clearly, if $G$ has  $n$ vertices, then $G$ 
has exactly $\binom{n}{2} - (n-5)$ edges. Hence the corresponding
graph $B$ has order $n$ and size $n-5$.  
Moreover, $n\ge 8$ by Lemma \ref{la:Garner-list}.

In this section we show that a minimal counterexample cannot have a reduction. 
We also show that no component of $B$ contains three independent vertices,
and that no component has two independent vertices that have at least two common
neighbors.

\begin{lemma}\label{lem:no-reduction} 
Let $G$ be a minimal counterexample. Then $B$ has no reduction. 
\end{lemma}

\begin{proof}
Suppose to the contrary that $B$ has a reduction $W$. 
Then $|W| >2$ and, by $m(B[W]) \geq |W|-1$, also $W \neq V$. 
Let $O_W$ be a non-trivial good orientation of $R[W]$ and 
let $U_1$ and $V_1$ be the partition classes of $O_W$. 
Create $B^{\star}$ from $B$ by removing the vertices of $W$ and adding two new vertines
$u_1,v_1$ with a blue edge $u_1v_1$ connecting them.
As $B[W]$ is a union
of components of $B$, $B$ contains no edges joining vertices in
$W$ to vertices in $V - W$. 
Then $n(B^{\star})=n+2-|W|<n$ and since $m(B[W]) \geq |W|-1$, 
\[ 1\leq m(B^{\star}) = (n-5)-m(B[W])+1 \leq   n-3 -|W|  < n(B^{\star})-5. \] 
In particular, $5<n(B^{\star})$.
Since $B$ was a minimal a counterexample, the red graph
$R^{\star}$ corresponding to $B^{\star}$ has an orientation $O^{\star}$ of diameter $2$. 

We now make use of $O_W$ and $O^{\star}$ to obtain an orientation $O_R$ of diameter $2$
of $R$. Let $x,y \in V$. If  $x, y\in W$ then orient $xy$ as
in $O_W$. If $x,y \in V-W$ then orient $xy$ as in $R^{\star}$. 
The remaining edges, joining a vertex in $x\in V-W$ to a vertex in $y \in W$ 
are oriented as follows. 
If $xu_1$ has received the orientation $\overrightarrow{xu_1}$
in $O^{\star}$ then we orient $x \rightarrow U_1$,  
and if $xu_1$ has received the orientation $\overrightarrow{u_1 x}$
in $O^{\star}$ then we orient $U_1 \rightarrow x$.  
Similarly, if $xv_1$ has received the orientation $\overrightarrow{xv_1}$
in $O^{\star}$ then we orient $x \rightarrow V_1$,  
and if $xv_1$ has received the orientation $\overrightarrow{v_1 x}$
in $O^{\star}$ then we orient $V_1 \rightarrow x$.  

If $x,y \in V$, then  either both vertices are in the same set $U_1$ (or $V_1$),
in which case there is a path of length at most two in $O_W$, or they are in 
different sets, for example $x\in U_1$ and $y\in V_1$, in which case the 
$(u_1,v_1)$-path in $O^{\star}$ gives rise to an $(x,y)$-path in $O_R$. 
\end{proof}

\begin{lemma} \label{la:important}
Let $G$ be a minimal counterexample. If  $X$ is an independent set of order
$3$ in $B$, and $N_i$ is the set of vertices in $v \in V-X$ having
exactly $i$ neighbors (in $B$) in $X$, then
\begin{equation} \label{eq:important}
|N_2| \leq 1 \quad \textrm{and} \quad N_3=\emptyset.
\end{equation}
\end{lemma}

\begin{proof} 
Suppose that $X=\{x_1,x_2,x_3\}$ is an independent set $B$
such that (\ref{eq:important}) does not hold. 
Create a new blue graph $B^{\star}$ by identifying the vertices of $X$ to a new vertex $x$ and removing multiple edges. Then $n(B^{\star})=n-2\ge 5$ and 
\[  m(B^{\star})  =  m(B)-|N_2| - 2|N_3| \leq m(B) - 2 = n-7 = n(B^{\star}) - 5. \]
Therefore, since $G$ is a minimal counterexample, the red graph 
$R^{\star}$ corresponding to $B^{\star}$ has an orientation $O^{\star}$ of diameter $2$. 

We will now orient $R$. 
Orient every edge $uv$ with $u,v \notin \{x_1, x_2, x_3\}$ as in $O^{\star}$. 
Orient $R[X]$ as $\overrightarrow{x_1x_2}$, $\overrightarrow{x_2x_3}$ and $\overrightarrow{x_3x_1}$,.
If an edge $ux$ is present in $R^{\star}$, then all edges $ux_i$, $i=1,2,3$ are present 
in $R$, and depending on whether $ux$ is oriented as $\overrightarrow{ux}$ 
or as $\overrightarrow{xu}$ in $O^{\star}$, we orient them 
$u \rightarrow \{x_1, x_2, x_3\}$ or $\{x_1, x_2, x_3\} \rightarrow u$. 
Orient any remaining edges in $R$ arbitrarily. to obtain the orientation $O_R$

Let $u,v\in V(G)$. 
If $u,v \in \{x_1, x_2, x_3\}$, then clearly there exists a $(u,v)$-path of length
at most two in $O_R$. 
If $u \in X$ and $v \in V-X$ or vice versa then the 
$(x,v)$-path of length at most two in $O^{\star}$ gives rise to a $(u,v)$-path of the same
length in $O_R$. 
If $u, v \in V-X$ then the $(u,v)$-path of length at most two in $O^{\star}$
gives rise to a $(u,v)$-path of the same length in $O_R$.  
This shows that $O_R$ is an orientation of $R$ of diameter $2$, a contradiction to 
$G$ being a counterexample. 
\end{proof}

\begin{lemma}\label{lem:alpha2} 
Let $G$ be a minimal counterexample.  Then no component of $B$ has 
three independent vertices. 
\end{lemma}

\begin{proof} 
Suppose to the contrary that $B$ has a component which contains three
independent vertices $x_1$, $x_2$ and $x_3$. We may assume that 
\begin{equation}    \label{eq:no-indep-set-1}
d_B(x_1,\{x_2, x_3\})=2.
\end{equation}  
Indeed, if $d_B(x_1, \{x_2, x_3\}) \geq 3$ then let $x_1^{\prime}$ be a 
vertex on a shortest path in $B$ from $x_1$ to $\{x_2,x_3\}$ that 
is at distance two from $\{x_2, x_3\}$. The new set $\{x_1^{\prime}, x_2, x_3\}$
is independent and satisfies \eqref{eq:no-indep-set-1}. \\
By \eqref{eq:no-indep-set-1} we may assume, possibly after renaming 
vertices, that $d_B(x_1,x_2)=2$. 
A similar argument as above now yields that we can choose $x_3$
such that also 
\begin{equation*}   
d_B(x_3,\{x_1, x_2\})=2.
\end{equation*}  
Hence we can choose $\{x_1, x_2, x_3\}$ such that it contains
at least two pairs of vertices at distance two in $B$. Hence,
possibly after renaming the vertices, we have 
\begin{equation} \label{eq:no-indep-set-3}
d_B(x_1, x_2) = d_B(x_2, x_3) =2. 
\end{equation}
Now \eqref{eq:no-indep-set-3} implies that there exists a common
neighbor $y_{12}$ of $x_1$ and $x_2$, and a common neighbor
$y_{23}$ of  $x_2$ and $x_3$ in $B$. If $y_{12} = y_{23}$, then
the set $N_3$ of vertices with exactly three neighbors in $\{x_1, x_2, x_3\}$ 
contains $y_{12}$ and is thus not empty, a contradiction to 
Lemma \ref{la:important}. If $y_{12} \neq y_{23}$, then
the
set $N_2$ of vertices with exactly two neighbors in $\{x_1, x_2, x_3\}$
contains $y_{12}$ and  $y_{23}$, again a contradiction to
Lemma \ref{la:important}. 
\end{proof}

\begin{lemma} \label{la:no-two-indep-vertices-share-two-neighbors}
Let $G$ be a minimal counterexample. If $x_1,x_2$ are independent vertices in $B$, then
$x_1$ and $x_2$ have at most one common blue neighbor. 
\end{lemma}

\begin{proof} 
Suppose to the contrary that $B$ has two vertices $x_1$ and $x_2$ that share at least two
neighbors. Then $x_1$ and $x_2$ are in the same component of $B$. 
As $n(B)>1$ and $m(B)<n(B)-1$, $B$ is not connected.
Choose a vertex $x_3$ from another component. 
Then $x_1,x_2,x_3$ are independent vertices, for which the set $N_2$ of 
vertices having exactly two neighbors in $\{x_1, x_2, x_3\}$ has 
at least two elements, a contradiction to Lemma \ref{la:important}. 
\end{proof}

\section{On tree components of $B$}
\label{sec:On Tree Components of $B$}

Since $B$ has $n$ vertices and $n-5$ edges, $B$ is not connected. In this section we
give useful lower bounds on the number of components of $B$ that are trees, and we show
that for a given order $t$ we can find a union $F_t$ of tree components of $B$ whose
order is close to $t$ and excess is at most $-t$. 
This will be useful in finding reductions and further restricting the possible structure of
$B$ for a minimal counterexample.
Recall that the {\em excess} of a graph $H$ is defined as 
${\rm ex}(H)= m(H) - n(H)$.

\begin{lemma}\label{lem:treecomp}  
If $B$ contains a component $B_1$ that is not a tree, then $B$ has
at least ${\rm ex}(B_1)+5\ge 5$ components that are trees. If $B$ has only tree components,
it has exactly five components. 
\end{lemma}

\begin{proof}
Let $T_1, T_2,\ldots,T_k$ be the components of $B$ that are trees, and 
\\$B_1, B_2, \ldots, B_{\ell}$ be the components that are not trees. Then 
${\rm ex}(T_i) =-1$ for all $i\in \{1,2,\ldots,k\}$ and
${\rm ex}(B_i)\geq 0$ for all $i\in \{1,2,\ldots,\ell\}$.
Since $m(B)= n - 5$, we have ${\rm ex}(B)=-5$, and so
\[ -5 = {\rm ex}(B) 
    = \sum_{i=1}^k {\rm ex}(T_i) + \sum_{i=1}^{\ell} {\rm ex}(B_i)
    = -k + \sum_{i=1}^{\ell} {\rm ex}(B_i)  \]
If $B$ has no tree component (i.e. $\ell=0$), this gives $k=5$. Hence, $B$ has exactly five components.  
If $B$ contains a component that is not a tree, $B_1$ say, then this
yields 
\[ -5  = -k + \sum_{i=1}^{\ell} {\rm ex}(B_i) 
     \geq  -k + {\rm ex}(B_1), \]
and so $k\geq 5+ {\rm ex}(B_1)\ge 5$, as claimed.      
\end{proof}

\begin{lemma}  \label{la:forest}
Assume $B$ contains at least
$t$ tree components whose size does not exceed $m_0$. 
Then there exists $t_0$ with $t \leq t_0 \leq t + m_0$ such 
that some subset of the tree components in $B$ forms a 
forest $F_t$ satisfying $n(F_t)=t_0$ and 
${\rm ex}(F_t) \geq -t$.
If $B$ contains a tree of size $m_0$ where $t>m_0$, then we can choose $F_t$ such that
${\rm ex}(F_t) \geq -t+m_0$. 
\end{lemma}

\begin{proof}  
Let $T_1, T_2, \ldots, T_t$ be the $t$ largest tree
components of $B$ whose size does not exceed $m_0$. 
Clearly $T_1 \cup T_2 \cup \cdots \cup T_t$ contains at least 
$t$ vertices. Let $j$ be the smallest positive integer 
such that $T_1 \cup T_2 \cup \cdots \cup T_j$ contains $t$ or
more vertices. Let $F_t=T_1 \cup T_2 \cup \cdots T_j$ 
and let $t_0=n(F_t)$.
Since $T_j$ has size at most $m_0$ and thus order at most $m_0+1$, we
have $t \leq t_0 \leq t+m_0$. Moreover, since 
$T_1 \cup T_2 \cup \cdots \cup T_{j-1}$ has less than $t$ vertices,
it follows that $T_j$ has at least $t_0-t+1$ vertices and
at least $t_0-t$ edges. Hence $m(F_t) \geq m(T_j) \geq t_0-t$,
and thus ${\rm ex}(F_t) \geq -t$. \\
If $t>m_0$, we have that $j\ge 2$ and $T_1$ has size
$m_0$. The same argument as above yields that
$m(F_t) \geq m(T_1) + m(T_j) = m_0 + t_0-t$ and thus
${\rm ex}(F_t) \geq -t+m_0$, as desired.  
\end{proof}

\section{Describing the components of $B$}
\label{section:describing the components of B}

In this section further restrict the structure of $B$ in a minimal
counterexample. We show that each component of $B$ is either a a path on at most 
four vertices, a complete graph, a proper dumbbell, a proper short dumbbell,
or a $5$-cycle, and none of these components have order more than six.

\begin{lemma} \label{la:C5-complete-dumbbell}
Let $G$ be a minimal counterexample and $B_1$ a component of $B$. 
\begin{enumerate}[{\rm (a)}]
\item If $B_1$ is a tree, then $B_1$ is a path $P_i$ with $1 \leq i \leq 4$. 
\item If $B_1$ is not a tree, then $B_1$ is one of the following:
\begin{enumerate}[{\rm (i)}]
\item a complete graph $K_i$ with $i\geq 3$,
\item a proper dumbbell,
\item a proper short dumbbell, or
\item a $5$-cycle.
\end{enumerate}
\end{enumerate}
\end{lemma}

\begin{proof}
As any tree that is not a path on at most $4$ vertices is not a complete graph, a dumbbell or a short dumbblell, it is enough to show that $B_1$ is a complete graph, 
a dumbbell,
a short dumbbell or a $5$ cycle.\\[1mm] 
If $B_1$ is complete, then the lemma holds, so assume that $B_1$ is not complete. 
Let $x_1$ and $x_2$ be two vertices of $B_1$ with 
$d_B(x_1, x_2) ={\rm diam}(B_1)\geq 2$. By Lemma~\ref{lem:alpha2} $B_1$ does not have three independent vertices, so
$d_B(x_1, x_2) ={\rm diam}(B_1)\leq 3$ and $V(B_1)=N_B(x_1)\cup N_B(x_2)$, and 
$|N_B(x_1)\cap N_B(x_2)|\leq 1$ by Lemma~\ref{la:no-two-indep-vertices-share-two-neighbors}. \\[1mm]
{\sc Case 1:} ${\rm diam}(B_1) = 3$ (consequently $N_B(x_1)\cap N_B(x_2)=\emptyset$). \\
Since $B_1$ does not have three independent vertices by Lemma \ref{lem:alpha2}, 
we conclude that each $N_B[x_i]$ forms a clique. \\
Since $B_1$ is connected, $B_1$ has an edge joining a vertex $y_1 \in N_B(x_1)$ to a 
vertex $y_2 \in N_B(x_2)$. We show that $B_1$ does not contain a further edge joining a vertex
$z_1\in N_B(x_1)$ to a vertex $z_2 \in N_B(x_2)$ by using that Lemma~\ref{la:no-two-indep-vertices-share-two-neighbors} gives
that two independent vertices share at most one neighbor. Indeed, if $y_1=z_1$, then 
$\{y_1, x_2\}$ would be a set of two independent vertices that share two neighbors.
If $y_2=z_2$, then 
$\{y_2, x_1\}$ would be a set of two independent vertices that share two neighbors.
Lastly, if $y_1 \neq z_1$ and $y_2\neq z_2$, then $\{y_1, z_2\}$ 
would be a set of two independent vertices that share two neighbors.
It follows that $B_1$ is a dumbbell.\\[1mm] 
{\sc Case 2:} ${\rm diam}(B_1) =2$ (consequently $N_B[x_1]\cap N_B[x_2]=\{y\}$) . \\
We consider two subcases: \\[1mm]
If $\min({\rm deg}_B(x_1),{\rm deg}_B(x_2))=1$, then
without loss of generality ${\rm deg}_B(x_1)=1$
and $N_B(x_1)=\{y\}$.  
Since ${\rm diam}(B_1)=2$, every vertex in $V(B_1) -\{x_1, y\}$ is 
adjacent to $y$ in $B_1$. Since $B_1$ does not contain three independent 
vertices, $V(B_1)-\{x_1, y\}$ induces a complete graph in $B_1$. 
Therefore $B_1$ is a short dumbbell. \\[1mm]
If $\min({\rm deg}_B(x_1),{\rm deg}_B(x_2))\geq 2$, then,
since $B_1$ does not contain three independent vertices, 
$N_B[x_i]\setminus\{y\}$ induces a complete graph in $B$ 
for $i\in\{1,2\}$. 
If $y$ is adjacent to all vertices in $B_1$, then 
$B_1$ is a short dumbbell and we are done. Assume without loss of generality 
that there is a vertex
$z_1\in N_B[x_1]$ to which $y$ is non-adjacent in  $B_1$. 
Then $d_B(z_1,x_2)=2$, so $z_1$ and $x_2$
have a common blue neighbor $z_2$. Since $x_1$ and $z_2$
are non-adjacent in $B$ and thus cannot have two common neighbors,
$z_2$ and $y$ are non-adjacent in $B$.  Since also the edges
$x_1x_2$, $x_1z_2$ and $x_2z_1$ are not present in $B$, we 
conclude that $x_1, y , x_2, z_2, z_1, x_1$ form an induced
$5$-cycle in $B_1$. Hence $B_1$ contains an induced $5$-cycle. \\
Rename the vertices of the $5$-cycle so the cycle is
$v_0v_1v_2v_3v_4v_0$. 
Suppose there is a sixth vertex $w$ adjacent to a vertex in 
$\{v_0, v_1, v_2, v_3, v_4\}$ in $B_1$. If $|N_B(w)\cap\{v_0,\ldots,v_4\}|\le 2$, it
is easy to see that $v$ together with two suitably chosen vertices
in $\{v_0, v_1, v_2, v_3, v_4\}$ forms an independent set of
cardinality three, which is impossible. Hence $v$ is adjacent 
to at least three vertices in $\{v_0, v_1, v_2, v_3, v_4\}$.
But then $v$ has two neighbors among these vertices that are
not adjacent, without loss of generality $v_1$ and $v_3$, so
that $v_1$ and $v_3$ are non-adjacent vertices with two common
neighbors, a contradiction to Lemma \ref{la:no-two-indep-vertices-share-two-neighbors}. This proves that $B_1$ contains only
$\{v_0, v_1, v_2, v_3, v_4\}$, and so $B_1$ is a $5$-cycle.  
\end{proof}

\begin{lemma} \label{la:no-large-component}
In a minimal counterexample all components of $B$ are of order at most six.
\end{lemma}

\begin{proof}
Suppose to the contrary that $B$ contains a component $B_1$ 
with more than six vertices. Let $n_1\ge 7$ and $m_1$ be the order
and size, respectively, of $B_1$. By Lemma \ref{la:C5-complete-dumbbell}, 
$B_1$ is a complete graph,
a dumbbell, or a short dumbbell. It is easy to see that among 
all such graphs of order $n_1$ the dumbbell 
$D_{\lceil n_1/2 \rceil, \lfloor n_1/2 \rfloor}$ has minimum size, and every other graph has bigger size. 
A simple calculation shows that 
\begin{equation} 
m_1 \geq m(D_{\lceil n_1/2 \rceil, \lfloor n_1/2 \rfloor})\ge \left\lceil \frac{1}{4}n_1^2 - \frac{1}{2}n_1 + 1 \right\rceil,
\end{equation}
and consequently
\begin{equation} \label{eq:no-large-component-excess}
{\rm ex}(B_1)= m_1-n_1\geq \left\lceil \frac{(n_1-3)^2-5}{4}\right\rceil, 
\end{equation}
where equality holds only when $B=D_{\lceil n_1/2 \rceil, \lfloor n_1/2 \rfloor}$.\\ [1mm]
Assume first that $B_1\ne D_{3,4}$.
If $n_1\ge 8$, equation \eqref{eq:no-large-component-excess} easily gives ${\rm ex}(B_1)\ge n_1-3$.
If $n_1=7$, then, as the lower bound in \eqref{eq:no-large-component-excess} is only sharp when $B_1=D_{3,4}$, it follows that
 ${\rm ex}(B_1) \geq 4=n_1-3$.
By Lemma \ref{lem:treecomp}, $B$ contains at least ${\rm ex}(B_1) + 5\ge n_1+2$ 
tree components. Set $t=n_1-2$. 
By Lemma~\ref{la:forest}, for some $t_0$ with $n_1-1 \leq t_0 \leq n_1+2$, 
$B$ contains a forest $F_{t}$ of order $t_0$ and excess at
least $-t=-n_1+2$ that is the union of the tree components of $B$. 
Let  $W:=V(B_1) \cup V(F_{t})$. We show that $W$ is a reduction. 
Clearly the graph $R[W]$ contains a
spanning subgraph $K_{n_1,t_0}$. Since $n_1-1 \leq t_0 \leq n_1+2$, 
it is easy to verify that  either $n_1 \leq t_0 \leq \binom{n_1}{2}$
or $t_0<n_1\leq \binom{t_0}{2}$. So $R[W]$
has a non-trivial good orientation by Lemma \ref{lem:Kab}, and
${\rm ex}(B[W]) =  {\rm ex}(B_1) + {\rm ex}(F_{n_1-2}) \geq   n_1-3 + (-n_1 +2) =-1$. Hence, $W$ is a reduction, a contradiction
to Lemma \ref{lem:no-reduction}. So we must have that $B_1=D_{3,4}$ \\[1mm]
If $B_1=D_{3,4}$,  ${\rm ex}(B_1)= 3$ by equation \eqref{eq:no-large-component-excess}, and
$B$ has at least $8$ tree components.
Set $m_0$ be the size of the largest tree component. If $1\le m_0$,   
Lemma \ref{la:forest} with $t=5$ and $1\leq m_0\leq 3$ yields that
there exists a forest $F_5$ in $B$ of order $t_0$, where $5\leq t_0 \leq 8$,
and excess at least $-5+1=-4$. Let $W=V(B_1) \cup V(F_5)$, then
${\rm ex}(B[W]) = {\rm ex}(B_1) + {\rm ex}(F_5) 
       \geq 3 + (-4) = -1$, and
$R[W]$ has a non-trivial good orientation by Lemma \ref{lem:Kab}.        
Hence $W$ is a reduction, a contradiction to Lemma \ref{lem:no-reduction}.  So all tree components of $B$ are singletons.
If all $k$ components of $B-B_1$ are $P_1$, then $-5={\rm ex}(B)=3-k$ gives
$B=D_{3,4}\cup 8P_1$.
But by Lemma \ref{la:Garner-list}, 
$\overline{D_{3,4}\cup 8P_1}$ has an orientation of diameter two, which is a contradiction.
Therefore $B$ contains another non-tree component $B_2$ with
at least one edge, so it has at least $3$ (and by our proof so far, at most $7$) vertices. Set $W=B_1\cup B_2\cup 2P_1$. By
Lemma~\ref{lem:Kab}, $W$ has a non-trivial good orientation with partition classes
$B_1$ and $B_2\cup 2P_1$ and ${\rm ex}(B[W])\ge 3-2>0$. So $W$ is a reduction, which is a contradiction.
\end{proof}

\begin{lemma} \label{la:no-two-non-trees}
If a minimal counterexample $B$ contains a component
$B_1$ that is not a tree, then $B-B_1$ has exactly  
${\rm ex}(B_1)+5$ components, all of which are trees. 
\end{lemma}

\begin{proof}
Suppose to the contrary that $B$ contains two non-tree components $B_1$ and $B_2$
with $3\le n(B_1)\le n(B_1)$. Then ${\rm ex}(B_1) \geq 0$ and ${\rm ex}(B_2) \geq 0$,
and by Lemma \ref{la:no-large-component}  $n(B_1) \leq 6$. 
 If $n(B_1)=n(B_2)=3$ or $n(B_1), n(B_2) \in \{4,5,6\}$, then 
$V(B_1) \cup V(B_2)$ has a non-trivial good orientation by Lemma \ref{lem:Kab}
and is thus a reduction, since ${\rm ex}(B_1 \cup B_2) = {\rm ex}(B_1) + {\rm ex}(B_2) \geq 0$. So we have
$n_1 \in \{4,5,6\}$ and $n_2=3$.
As $V(B_1) \cup V(P_4)$
or $V(B_2) \cup V(P_3)$ would form a reduction, all tree components in $B$ are $P_1$ or $P_2$.
Since $V(B_1) \cup V(B_2) \cup V(P_i)$ forms a reduction for $i\in\{1,2\}$,
this is a contradiction to Lemma \ref{lem:no-reduction}.
Hence all $k$ components of $B-B_1$ are trees. 
As $-5 = {\rm ex}(B) =  {\rm ex}(B_1) - k$
we are done.
\end{proof}

\begin{lemma}\label{la:small-trees}
Assume $B$ contains a non-tree component $B_1$. Let $F$ be a forest that is the union of
the smallest number of tree components of $B$ such that $\min(4,n(B_1))\le n(F)\le 6$ and $k_0$ be the number of tree components that make up $F$.
Then $k_0\geq {\rm ex}(B_1)+2$, ${\rm ex}(B_1)\le 2$, and the tree components of $B$ contain at most  $\min(3,n(B_1)-1)$ vertices.
\end{lemma}\

\begin{proof} If $B_1$ is a component that is not a tree, by Lemma~\ref{la:no-large-component} ${\rm ex}(B_1)\ge 0$ and $3\le n(B_1)\le 6$.
$B$ does not contain a $P_4$ component, otherwise
$W=V(B_1)\cup V(P_4)$ would form a reduction
by Lemma~\ref{lem:Kab} and ${\rm ex}(B[W])\ge -1$.

Let $F$ and $k_0$ be given as in the conditions of the lemma. Clearly, $k_0\le 4$
and ${\rm ex}(F)=-k_0$. Consider $W=V(B_1)\cup V(F)$. If $n(B_1)\ne 3$ or $n(B_1)=n(F)$, then $R[W]$ has a non-trivial good orientation by Lemma~\ref{lem:Kab}.
If $n(B_1)=3$ and $4\le n(F)\le 6$, then $B_1=K_3$, $F\trianglelefteq K_3\boxplus K_{n(F)-3}$, and $R[W]$ has a non-trivial good orientation by
Lemma~\ref{lem:Kabsubgraph}. As $W$ is not a reduction, we must have
$-2\ge {\rm ex}(B[W])={\rm ex}(B_1)-k$, giving $k\ge {\rm ex}(B_1)+2$. ${\rm ex}(B_1)\le 2$ follows from $k\le 4$. As $k\ge 2$, no tree component has size $n(B_1)$.
\end{proof}

\section{Proof of the main result}
\label{sec:main-result}

We start by eliminating the possibility of a non-tree component from a minimal counterexample.

\begin{lemma} \label{la:complete-graphs}
In a minimal counterexample no component of $B$ is a complete graph on three or more vertices.  
\end{lemma}

\begin{proof}
Suppose to the contrary that $B$ contains a component $B_1$ that is a complete
graph of order $n_1 \geq 3$. By Lemma~\ref{la:small-trees} we have ${\rm ex}(B_1)\le 2$ and consequently $n_1\in\{3,4\}$.\\[1mm]
 If  $B_1=K_4$, then ${\rm ex}(B_1)=2$ and $B$ contains exactly $7$ tree components by Lemma \ref{la:no-two-non-trees}. 
By Lemma~\ref{la:small-trees} all tree components must be $P_1$ (otherwise $k_0<4$ in the lemma, which is a contradiction).
By Lemma \ref{la:Garner-list}, the graph $\overline{K_4 \cup 7K_1}$ has an orientation of diameter two, so it is not a counterexample, which is a contradiction.\\[1mm]
If  $B_1=K_3$, then ${\rm ex}(B_1)=0$ and $B$ contains exactly $5$ tree components by Lemma \ref{la:no-two-non-trees}. 
By Lemma~\ref{la:small-trees} all these tree components must be $P_1$ or $P_2$, so
we have $B=K_3 \cup aK_1 \cup bK_2$
for some nonnegative integers $a,b$ with $a+b=5$. But by Lemma 
\ref{la:Garner-list} all such graphs have an orientation of diameter two. So $G$ is not a counterexample, a contradiction.  
\end{proof}

\begin{lemma} \label{la:dumbbells}
In a minimal counterexample no component of $B$ is a proper dumbbell. 
\end{lemma}

\begin{proof}
Assume that $B_1$ is a component of $B$ that is a proper dumbbell; then $n(B_1) \geq 4$,
and by Lemmata~\ref{la:no-large-component} and ~\ref{la:small-trees} we have $n(B_1)\leq 6$ and ${\rm ex}(B_1)\le 2$, and $B-B_1$ has no $P_4$ component.
Hence
$B_1\in\{D_{3,1},D_{4,1},D_{3,2},D_{4,2},D_{3,3}\}$. By Lemma \ref{la:no-two-non-trees}, $B-B_1$ has
exactly ${\rm ex}(B_1)+5$ other components that are all paths on at most three vertices.
We will examine each case grouped by ${\rm ex}(B_1)$.
\begin{enumerate}
\item $B_1\in\{D_{4,1},D_{4,2}\}$.
Then ${\rm ex}(B_1)=2$ and $B-B_1$ has exactly $7$ tree components 
by Lemma \ref{la:no-two-non-trees}. 
By Lemma \ref{la:Garner-list}, the graph $\overline{B_1\cup 7K_1}$ has an orientation of diameter two, 
therefore 
not all tree components of $B$ are singletons. We get $k_0\le 3$ and a contradiction in Lemma~\ref{la:small-trees}.
\item$B_1\in\{D_{3,1},D_{3,2}\}$. 
Then ${\rm ex}(B_1)=0$ and  $B-B_1$ has exactly five
tree components by Lemma \ref{la:no-two-non-trees}. 
Lemma~\ref{la:Garner-list} gives that $\overline{B_1 \cup aK_1 \cup bK_2}$ has a diameter two orientation
for all 
$a+b=5$, so at least one of the tree components is a $P_3$ . 
For $j\in\{1,2\}$, $D_{3,j}\trianglelefteq K_3\boxplus K_j$,
and by Lemma~\ref{lem:Kabsubgraph} $V(P_3)\cup V(B_1)$ is a reduction, which is again a contradiction.
\item $B_1=D_{3,3}$. Then ${\rm ex}(B_1)=1$.
By Lemma \ref{la:no-two-non-trees}, $B-B_1$ contains exactly $6$
components which are trees. By Lemma \ref{la:Garner-list}, the graphs
$\overline{D_{3,3} \cup 6K_1}$ and
$\overline{D_{3,3} \cup K_2 \cup 5K_1}$ have an orientation of diameter two. Hence $B-B_1$
contains a $P_3$ or two components that are $P_2$. We get  $k_0\le 2$ and a contradiction in Lemma~\ref{la:small-trees}.
\end{enumerate}
\end{proof}

\begin{lemma} \label{la:short-dumbbells}
In a minimal counterexample no component of $B$ is a proper short dumbbell. 
\end{lemma}

\begin{proof}
Assume that $B_1$ is a component of $B$ that is  a proper short dumbbell.
Then $5\le n(B_1)$. By Lemmata \ref{la:no-large-component} and ~\ref{la:small-trees}, 
$n(B_1) \leq 6$, ${\rm ex}(B_1)\le 2$, and no tree component of $B$ is a $P_4$. This gives that $B_1=S_{3,3}$, ${\rm ex}(B_1)=1$, and $B-B_1$ has exactly $6$ tree components. By  Lemma \ref{la:Garner-list}, both  $\overline{S_{3,3} \cup 6K_1}$ and
$\overline{S_{3,3} \cup K_2 \cup 5K_1}$ have diameter two orientations, so
 the components of $B$ include at least two $P_2$ or at least one $P_3$. This gives $k_0=2$ and a contradiction in Lemma~\ref{la:small-trees}.
\end{proof}

\begin{lemma} \label{la:5-cycle}
In a minimal counterexample 
no component of $B$ is a $5$-cycle.  
\end{lemma}

\begin{proof}
Assume that $B_1$ is a component of $B$ that is
a $5$-cycle. Then ${\rm ex}(B_1)=0$ and, by Lemmata \ref{la:no-two-non-trees} and~\ref{la:small-trees},
$B-B_1$ has exactly $5$ components which are trees on at most three vertices.
By Lemma \ref{la:Garner-list}, $\overline{C_5 \cup aP_2 \cup bP_1}$ has an orientation
of diameter two for all non-negative integers 
$a,b$ with $a+b=5$, so at least one of these tree components is a $P_3$.
As $P_3\leq K_3$ and $C_5\trianglelefteq K_3\boxplus K_2$, by Lemma \ref{lem:Kabsubgraph} $B_1\cup P_3$ forms a reduction, 
contradicting Lemma \ref{lem:no-reduction}.
\end{proof}

We are now ready to complete the proof of Theorem \ref{theo:main-result}. 

\begin{proof}
Suppose to the contrary that Theorem \ref{theo:main-result} is false.
Let $G$ be a minimal counterexample, that is a graph of minimum order
and minimum size for which the theorem does not hold. By Lemma \ref{la:Garner-list}, 
$n(G)\ge 8$ and
consequently $m(G)=n(G)-5$. 
By Lemma \ref{la:C5-complete-dumbbell},  every component
of $B$ that is not a tree is either a complete graph on at least three vertices,
a proper dumbbell, a proper short dumbbell, or a $5$-cycle.  By Lemmata~\ref{la:complete-graphs}, \ref{la:dumbbells}, \ref{la:short-dumbbells}, and ~\ref{la:5-cycle},
all components of $B$ must be trees, and by Lemmata~\ref{lem:treecomp} and ~\ref{la:C5-complete-dumbbell}
$B=aP_1\cup bP_2\cup cP_3\cup dP_4$ for some $a+b+c+d=5$.
But then Lemma \ref{la:Garner-list} gives that $G$ has a diameter two orientation, a contradicton.
\end{proof}

\section{Open Problem}

In Theorem \ref{theo:main-result}, we show that in graph of given order $n$ we need at least $\binom{n}{2}-n+5$ edges to guarantee the existence of an orientation of diameter two. It is natural to ask the same question for any given value of $d$: In a graph of order $n$, over all bridgeless graphs, how many edges do we need at least to guarantee the existence of an orientation of diameter at most $d$?

\end{document}

%% file: AllCases.tikz

\begin{tikzpicture}
\definecolor{cv0}{rgb}{0.0,0.0,0.0}
\definecolor{cfv0}{rgb}{1.0,1.0,1.0}
\definecolor{clv0}{rgb}{0.0,0.0,0.0}
\definecolor{cv1}{rgb}{0.0,0.0,0.0}
\definecolor{cfv1}{rgb}{1.0,1.0,1.0}
\definecolor{clv1}{rgb}{0.0,0.0,0.0}
\definecolor{cv2}{rgb}{0.0,0.0,0.0}
\definecolor{cfv2}{rgb}{1.0,1.0,1.0}
\definecolor{clv2}{rgb}{0.0,0.0,0.0}
\definecolor{cv3}{rgb}{0.0,0.0,0.0}
\definecolor{cfv3}{rgb}{1.0,1.0,1.0}
\definecolor{clv3}{rgb}{0.0,0.0,0.0}
\definecolor{cv4}{rgb}{0.0,0.0,0.0}
\definecolor{cfv4}{rgb}{1.0,1.0,1.0}
\definecolor{clv4}{rgb}{0.0,0.0,0.0}
\definecolor{cv0v1}{rgb}{0.0,0.0,0.0}
\definecolor{cv0v2}{rgb}{0.0,0.0,0.0}
\definecolor{cv0v3}{rgb}{0.0,0.0,0.0}
\definecolor{cv1v2}{rgb}{0.0,0.0,0.0}
\definecolor{cv1v4}{rgb}{0.0,0.0,0.0}
\definecolor{cv2v3}{rgb}{0.0,0.0,0.0}
\definecolor{cv2v4}{rgb}{0.0,0.0,0.0}
\definecolor{cv3v1}{rgb}{0.0,0.0,0.0}
\definecolor{cv3v4}{rgb}{0.0,0.0,0.0}
\definecolor{cv4v0}{rgb}{0.0,0.0,0.0}
\Vertex[style={minimum size=1.0cm,draw=cv0,fill=cfv0,text=clv0,shape=circle},LabelOut=false,L=\hbox{$0$},x=2.5cm,y=5.0cm]{v0}
\Vertex[style={minimum size=1.0cm,draw=cv1,fill=cfv1,text=clv1,shape=circle},LabelOut=false,L=\hbox{$1$},x=0.0cm,y=3.0902cm]{v1}
\Vertex[style={minimum size=1.0cm,draw=cv2,fill=cfv2,text=clv2,shape=circle},LabelOut=false,L=\hbox{$2$},x=0.9549cm,y=0.0cm]{v2}
\Vertex[style={minimum size=1.0cm,draw=cv3,fill=cfv3,text=clv3,shape=circle},LabelOut=false,L=\hbox{$3$},x=4.0451cm,y=0.0cm]{v3}
\Vertex[style={minimum size=1.0cm,draw=cv4,fill=cfv4,text=clv4,shape=circle},LabelOut=false,L=\hbox{$4$},x=5.0cm,y=3.0902cm]{v4}
\Edge[lw=0.05cm,style={post,color=cv0v1,},](v0)(v1)
\Edge[lw=0.05cm,style={post,color=cv0v2,},](v0)(v2)
\Edge[lw=0.05cm,style={post,color=cv0v3,},](v0)(v3)
\Edge[lw=0.05cm,style={post,color=cv1v2,},](v1)(v2)
\Edge[lw=0.05cm,style={post,color=cv1v4,},](v1)(v4)
\Edge[lw=0.05cm,style={post,color=cv2v3,},](v2)(v3)
\Edge[lw=0.05cm,style={post,color=cv2v4,},](v2)(v4)
\Edge[lw=0.05cm,style={post,color=cv3v1,},](v3)(v1)
\Edge[lw=0.05cm,style={post,color=cv3v4,},](v3)(v4)
\Edge[lw=0.05cm,style={post,color=cv4v0,},](v4)(v0)
\end{tikzpicture}
\begin{tikzpicture}
\definecolor{cv0}{rgb}{0.0,0.0,0.0}
\definecolor{cfv0}{rgb}{1.0,1.0,1.0}
\definecolor{clv0}{rgb}{0.0,0.0,0.0}
\definecolor{cv1}{rgb}{0.0,0.0,0.0}
\definecolor{cfv1}{rgb}{1.0,1.0,1.0}
\definecolor{clv1}{rgb}{0.0,0.0,0.0}
\definecolor{cv2}{rgb}{0.0,0.0,0.0}
\definecolor{cfv2}{rgb}{1.0,1.0,1.0}
\definecolor{clv2}{rgb}{0.0,0.0,0.0}
\definecolor{cv3}{rgb}{0.0,0.0,0.0}
\definecolor{cfv3}{rgb}{1.0,1.0,1.0}
\definecolor{clv3}{rgb}{0.0,0.0,0.0}
\definecolor{cv4}{rgb}{0.0,0.0,0.0}
\definecolor{cfv4}{rgb}{1.0,1.0,1.0}
\definecolor{clv4}{rgb}{0.0,0.0,0.0}
\definecolor{cv5}{rgb}{0.0,0.0,0.0}
\definecolor{cfv5}{rgb}{1.0,1.0,1.0}
\definecolor{clv5}{rgb}{0.0,0.0,0.0}
\definecolor{cv0v2}{rgb}{0.0,0.0,0.0}
\definecolor{cv0v4}{rgb}{0.0,0.0,0.0}
\definecolor{cv1v2}{rgb}{0.0,0.0,0.0}
\definecolor{cv1v3}{rgb}{0.0,0.0,0.0}
\definecolor{cv2v3}{rgb}{0.0,0.0,0.0}
\definecolor{cv2v4}{rgb}{0.0,0.0,0.0}
\definecolor{cv2v5}{rgb}{0.0,0.0,0.0}
\definecolor{cv3v0}{rgb}{0.0,0.0,0.0}
\definecolor{cv3v4}{rgb}{0.0,0.0,0.0}
\definecolor{cv3v5}{rgb}{0.0,0.0,0.0}
\definecolor{cv4v1}{rgb}{0.0,0.0,0.0}
\definecolor{cv4v5}{rgb}{0.0,0.0,0.0}
\definecolor{cv5v0}{rgb}{0.0,0.0,0.0}
\definecolor{cv5v1}{rgb}{0.0,0.0,0.0}
\Vertex[style={minimum size=1.0cm,draw=cv0,fill=cfv0,text=clv0,shape=circle},LabelOut=false,L=\hbox{$0$},x=2.5cm,y=5.0cm]{v0}
\Vertex[style={minimum size=1.0cm,draw=cv1,fill=cfv1,text=clv1,shape=circle},LabelOut=false,L=\hbox{$1$},x=0.0cm,y=3.75cm]{v1}
\Vertex[style={minimum size=1.0cm,draw=cv2,fill=cfv2,text=clv2,shape=circle},LabelOut=false,L=\hbox{$2$},x=0.0cm,y=1.25cm]{v2}
\Vertex[style={minimum size=1.0cm,draw=cv3,fill=cfv3,text=clv3,shape=circle},LabelOut=false,L=\hbox{$3$},x=2.5cm,y=0.0cm]{v3}
\Vertex[style={minimum size=1.0cm,draw=cv4,fill=cfv4,text=clv4,shape=circle},LabelOut=false,L=\hbox{$4$},x=5.0cm,y=1.25cm]{v4}
\Vertex[style={minimum size=1.0cm,draw=cv5,fill=cfv5,text=clv5,shape=circle},LabelOut=false,L=\hbox{$5$},x=5.0cm,y=3.75cm]{v5}
\Edge[lw=0.05cm,style={post,color=cv0v2,},](v0)(v2)
\Edge[lw=0.05cm,style={post,color=cv0v4,},](v0)(v4)
\Edge[lw=0.05cm,style={post,color=cv1v2,},](v1)(v2)
\Edge[lw=0.05cm,style={post,color=cv1v3,},](v1)(v3)
\Edge[lw=0.05cm,style={post,color=cv2v3,},](v2)(v3)
\Edge[lw=0.05cm,style={post,color=cv2v4,},](v2)(v4)
\Edge[lw=0.05cm,style={post,color=cv2v5,},](v2)(v5)
\Edge[lw=0.05cm,style={post,color=cv3v0,},](v3)(v0)
\Edge[lw=0.05cm,style={post,color=cv3v4,},](v3)(v4)
\Edge[lw=0.05cm,style={post,color=cv3v5,},](v3)(v5)
\Edge[lw=0.05cm,style={post,color=cv4v1,},](v4)(v1)
\Edge[lw=0.05cm,style={post,color=cv4v5,},](v4)(v5)
\Edge[lw=0.05cm,style={post,color=cv5v0,},](v5)(v0)
\Edge[lw=0.05cm,style={post,color=cv5v1,},](v5)(v1)
\end{tikzpicture}


\begin{tikzpicture}
\definecolor{cv0}{rgb}{0.0,0.0,0.0}
\definecolor{cfv0}{rgb}{1.0,1.0,1.0}
\definecolor{clv0}{rgb}{0.0,0.0,0.0}
\definecolor{cv1}{rgb}{0.0,0.0,0.0}
\definecolor{cfv1}{rgb}{1.0,1.0,1.0}
\definecolor{clv1}{rgb}{0.0,0.0,0.0}
\definecolor{cv2}{rgb}{0.0,0.0,0.0}
\definecolor{cfv2}{rgb}{1.0,1.0,1.0}
\definecolor{clv2}{rgb}{0.0,0.0,0.0}
\definecolor{cv3}{rgb}{0.0,0.0,0.0}
\definecolor{cfv3}{rgb}{1.0,1.0,1.0}
\definecolor{clv3}{rgb}{0.0,0.0,0.0}
\definecolor{cv4}{rgb}{0.0,0.0,0.0}
\definecolor{cfv4}{rgb}{1.0,1.0,1.0}
\definecolor{clv4}{rgb}{0.0,0.0,0.0}
\definecolor{cv5}{rgb}{0.0,0.0,0.0}
\definecolor{cfv5}{rgb}{1.0,1.0,1.0}
\definecolor{clv5}{rgb}{0.0,0.0,0.0}
\definecolor{cv6}{rgb}{0.0,0.0,0.0}
\definecolor{cfv6}{rgb}{1.0,1.0,1.0}
\definecolor{clv6}{rgb}{0.0,0.0,0.0}
\definecolor{cv0v2}{rgb}{0.0,0.0,0.0}
\definecolor{cv0v4}{rgb}{0.0,0.0,0.0}
\definecolor{cv0v5}{rgb}{0.0,0.0,0.0}
\definecolor{cv1v3}{rgb}{0.0,0.0,0.0}
\definecolor{cv1v4}{rgb}{0.0,0.0,0.0}
\definecolor{cv2v3}{rgb}{0.0,0.0,0.0}
\definecolor{cv2v6}{rgb}{0.0,0.0,0.0}
\definecolor{cv3v0}{rgb}{0.0,0.0,0.0}
\definecolor{cv3v4}{rgb}{0.0,0.0,0.0}
\definecolor{cv3v5}{rgb}{0.0,0.0,0.0}
\definecolor{cv3v6}{rgb}{0.0,0.0,0.0}
\definecolor{cv4v2}{rgb}{0.0,0.0,0.0}
\definecolor{cv4v5}{rgb}{0.0,0.0,0.0}
\definecolor{cv4v6}{rgb}{0.0,0.0,0.0}
\definecolor{cv5v1}{rgb}{0.0,0.0,0.0}
\definecolor{cv5v2}{rgb}{0.0,0.0,0.0}
\definecolor{cv5v6}{rgb}{0.0,0.0,0.0}
\definecolor{cv6v0}{rgb}{0.0,0.0,0.0}
\definecolor{cv6v1}{rgb}{0.0,0.0,0.0}
\Vertex[style={minimum size=1.0cm,draw=cv0,fill=cfv0,text=clv0,shape=circle},LabelOut=false,L=\hbox{$0$},x=2.5cm,y=5.0cm]{v0}
\Vertex[style={minimum size=1.0cm,draw=cv1,fill=cfv1,text=clv1,shape=circle},LabelOut=false,L=\hbox{$1$},x=0.4952cm,y=4.0097cm]{v1}
\Vertex[style={minimum size=1.0cm,draw=cv2,fill=cfv2,text=clv2,shape=circle},LabelOut=false,L=\hbox{$2$},x=0.0cm,y=1.7845cm]{v2}
\Vertex[style={minimum size=1.0cm,draw=cv3,fill=cfv3,text=clv3,shape=circle},LabelOut=false,L=\hbox{$3$},x=1.3874cm,y=0.0cm]{v3}
\Vertex[style={minimum size=1.0cm,draw=cv4,fill=cfv4,text=clv4,shape=circle},LabelOut=false,L=\hbox{$4$},x=3.6126cm,y=0.0cm]{v4}
\Vertex[style={minimum size=1.0cm,draw=cv5,fill=cfv5,text=clv5,shape=circle},LabelOut=false,L=\hbox{$5$},x=5.0cm,y=1.7845cm]{v5}
\Vertex[style={minimum size=1.0cm,draw=cv6,fill=cfv6,text=clv6,shape=circle},LabelOut=false,L=\hbox{$6$},x=4.5048cm,y=4.0097cm]{v6}
\Edge[lw=0.05cm,style={post,color=cv0v2,},](v0)(v2)
\Edge[lw=0.05cm,style={post,color=cv0v4,},](v0)(v4)
\Edge[lw=0.05cm,style={post,color=cv0v5,},](v0)(v5)
\Edge[lw=0.05cm,style={post,color=cv1v3,},](v1)(v3)
\Edge[lw=0.05cm,style={post,color=cv1v4,},](v1)(v4)
\Edge[lw=0.05cm,style={post,color=cv2v3,},](v2)(v3)
\Edge[lw=0.05cm,style={post,color=cv2v6,},](v2)(v6)
\Edge[lw=0.05cm,style={post,color=cv3v0,},](v3)(v0)
\Edge[lw=0.05cm,style={post,color=cv3v4,},](v3)(v4)
\Edge[lw=0.05cm,style={post,color=cv3v5,},](v3)(v5)
\Edge[lw=0.05cm,style={post,color=cv3v6,},](v3)(v6)
\Edge[lw=0.05cm,style={post,color=cv4v2,},](v4)(v2)
\Edge[lw=0.05cm,style={post,color=cv4v5,},](v4)(v5)
\Edge[lw=0.05cm,style={post,color=cv4v6,},](v4)(v6)
\Edge[lw=0.05cm,style={post,color=cv5v1,},](v5)(v1)
\Edge[lw=0.05cm,style={post,color=cv5v2,},](v5)(v2)
\Edge[lw=0.05cm,style={post,color=cv5v6,},](v5)(v6)
\Edge[lw=0.05cm,style={post,color=cv6v0,},](v6)(v0)
\Edge[lw=0.05cm,style={post,color=cv6v1,},](v6)(v1)
\end{tikzpicture}
\begin{tikzpicture}
\definecolor{cv0}{rgb}{0.0,0.0,0.0}
\definecolor{cfv0}{rgb}{1.0,1.0,1.0}
\definecolor{clv0}{rgb}{0.0,0.0,0.0}
\definecolor{cv1}{rgb}{0.0,0.0,0.0}
\definecolor{cfv1}{rgb}{1.0,1.0,1.0}
\definecolor{clv1}{rgb}{0.0,0.0,0.0}
\definecolor{cv2}{rgb}{0.0,0.0,0.0}
\definecolor{cfv2}{rgb}{1.0,1.0,1.0}
\definecolor{clv2}{rgb}{0.0,0.0,0.0}
\definecolor{cv3}{rgb}{0.0,0.0,0.0}
\definecolor{cfv3}{rgb}{1.0,1.0,1.0}
\definecolor{clv3}{rgb}{0.0,0.0,0.0}
\definecolor{cv4}{rgb}{0.0,0.0,0.0}
\definecolor{cfv4}{rgb}{1.0,1.0,1.0}
\definecolor{clv4}{rgb}{0.0,0.0,0.0}
\definecolor{cv5}{rgb}{0.0,0.0,0.0}
\definecolor{cfv5}{rgb}{1.0,1.0,1.0}
\definecolor{clv5}{rgb}{0.0,0.0,0.0}
\definecolor{cv6}{rgb}{0.0,0.0,0.0}
\definecolor{cfv6}{rgb}{1.0,1.0,1.0}
\definecolor{clv6}{rgb}{0.0,0.0,0.0}
\definecolor{cv0v2}{rgb}{0.0,0.0,0.0}
\definecolor{cv0v3}{rgb}{0.0,0.0,0.0}
\definecolor{cv1v3}{rgb}{0.0,0.0,0.0}
\definecolor{cv1v4}{rgb}{0.0,0.0,0.0}
\definecolor{cv1v5}{rgb}{0.0,0.0,0.0}
\definecolor{cv2v1}{rgb}{0.0,0.0,0.0}
\definecolor{cv2v4}{rgb}{0.0,0.0,0.0}
\definecolor{cv2v6}{rgb}{0.0,0.0,0.0}
\definecolor{cv3v4}{rgb}{0.0,0.0,0.0}
\definecolor{cv3v5}{rgb}{0.0,0.0,0.0}
\definecolor{cv3v6}{rgb}{0.0,0.0,0.0}
\definecolor{cv4v0}{rgb}{0.0,0.0,0.0}
\definecolor{cv4v5}{rgb}{0.0,0.0,0.0}
\definecolor{cv4v6}{rgb}{0.0,0.0,0.0}
\definecolor{cv5v0}{rgb}{0.0,0.0,0.0}
\definecolor{cv5v2}{rgb}{0.0,0.0,0.0}
\definecolor{cv5v6}{rgb}{0.0,0.0,0.0}
\definecolor{cv6v0}{rgb}{0.0,0.0,0.0}
\definecolor{cv6v1}{rgb}{0.0,0.0,0.0}
\Vertex[style={minimum size=1.0cm,draw=cv0,fill=cfv0,text=clv0,shape=circle},LabelOut=false,L=\hbox{$0$},x=2.5cm,y=5.0cm]{v0}
\Vertex[style={minimum size=1.0cm,draw=cv1,fill=cfv1,text=clv1,shape=circle},LabelOut=false,L=\hbox{$1$},x=0.4952cm,y=4.0097cm]{v1}
\Vertex[style={minimum size=1.0cm,draw=cv2,fill=cfv2,text=clv2,shape=circle},LabelOut=false,L=\hbox{$2$},x=0.0cm,y=1.7845cm]{v2}
\Vertex[style={minimum size=1.0cm,draw=cv3,fill=cfv3,text=clv3,shape=circle},LabelOut=false,L=\hbox{$3$},x=1.3874cm,y=0.0cm]{v3}
\Vertex[style={minimum size=1.0cm,draw=cv4,fill=cfv4,text=clv4,shape=circle},LabelOut=false,L=\hbox{$4$},x=3.6126cm,y=0.0cm]{v4}
\Vertex[style={minimum size=1.0cm,draw=cv5,fill=cfv5,text=clv5,shape=circle},LabelOut=false,L=\hbox{$5$},x=5.0cm,y=1.7845cm]{v5}
\Vertex[style={minimum size=1.0cm,draw=cv6,fill=cfv6,text=clv6,shape=circle},LabelOut=false,L=\hbox{$6$},x=4.5048cm,y=4.0097cm]{v6}
\Edge[lw=0.05cm,style={post,color=cv0v2,},](v0)(v2)
\Edge[lw=0.05cm,style={post,color=cv0v3,},](v0)(v3)
\Edge[lw=0.05cm,style={post,color=cv1v3,},](v1)(v3)
\Edge[lw=0.05cm,style={post,color=cv1v4,},](v1)(v4)
\Edge[lw=0.05cm,style={post,color=cv1v5,},](v1)(v5)
\Edge[lw=0.05cm,style={post,color=cv2v1,},](v2)(v1)
\Edge[lw=0.05cm,style={post,color=cv2v4,},](v2)(v4)
\Edge[lw=0.05cm,style={post,color=cv2v6,},](v2)(v6)
\Edge[lw=0.05cm,style={post,color=cv3v4,},](v3)(v4)
\Edge[lw=0.05cm,style={post,color=cv3v5,},](v3)(v5)
\Edge[lw=0.05cm,style={post,color=cv3v6,},](v3)(v6)
\Edge[lw=0.05cm,style={post,color=cv4v0,},](v4)(v0)
\Edge[lw=0.05cm,style={post,color=cv4v5,},](v4)(v5)
\Edge[lw=0.05cm,style={post,color=cv4v6,},](v4)(v6)
\Edge[lw=0.05cm,style={post,color=cv5v0,},](v5)(v0)
\Edge[lw=0.05cm,style={post,color=cv5v2,},](v5)(v2)
\Edge[lw=0.05cm,style={post,color=cv5v6,},](v5)(v6)
\Edge[lw=0.05cm,style={post,color=cv6v0,},](v6)(v0)
\Edge[lw=0.05cm,style={post,color=cv6v1,},](v6)(v1)
\end{tikzpicture}


\begin{center}

\begin{tikzpicture}
\definecolor{cv0}{rgb}{0.0,0.0,0.0}
\definecolor{cfv0}{rgb}{1.0,1.0,1.0}
\definecolor{clv0}{rgb}{0.0,0.0,0.0}
\definecolor{cv1}{rgb}{0.0,0.0,0.0}
\definecolor{cfv1}{rgb}{1.0,1.0,1.0}
\definecolor{clv1}{rgb}{0.0,0.0,0.0}
\definecolor{cv2}{rgb}{0.0,0.0,0.0}
\definecolor{cfv2}{rgb}{1.0,1.0,1.0}
\definecolor{clv2}{rgb}{0.0,0.0,0.0}
\definecolor{cv3}{rgb}{0.0,0.0,0.0}
\definecolor{cfv3}{rgb}{1.0,1.0,1.0}
\definecolor{clv3}{rgb}{0.0,0.0,0.0}
\definecolor{cv4}{rgb}{0.0,0.0,0.0}
\definecolor{cfv4}{rgb}{1.0,1.0,1.0}
\definecolor{clv4}{rgb}{0.0,0.0,0.0}
\definecolor{cv5}{rgb}{0.0,0.0,0.0}
\definecolor{cfv5}{rgb}{1.0,1.0,1.0}
\definecolor{clv5}{rgb}{0.0,0.0,0.0}
\definecolor{cv6}{rgb}{0.0,0.0,0.0}
\definecolor{cfv6}{rgb}{1.0,1.0,1.0}
\definecolor{clv6}{rgb}{0.0,0.0,0.0}
\definecolor{cv7}{rgb}{0.0,0.0,0.0}
\definecolor{cfv7}{rgb}{1.0,1.0,1.0}
\definecolor{clv7}{rgb}{0.0,0.0,0.0}
\definecolor{cv0v3}{rgb}{0.0,0.0,0.0}
\definecolor{cv0v4}{rgb}{0.0,0.0,0.0}
\definecolor{cv0v7}{rgb}{0.0,0.0,0.0}
\definecolor{cv1v3}{rgb}{0.0,0.0,0.0}
\definecolor{cv1v6}{rgb}{0.0,0.0,0.0}
\definecolor{cv2v0}{rgb}{0.0,0.0,0.0}
\definecolor{cv2v4}{rgb}{0.0,0.0,0.0}
\definecolor{cv2v5}{rgb}{0.0,0.0,0.0}
\definecolor{cv3v4}{rgb}{0.0,0.0,0.0}
\definecolor{cv3v5}{rgb}{0.0,0.0,0.0}
\definecolor{cv3v6}{rgb}{0.0,0.0,0.0}
\definecolor{cv3v7}{rgb}{0.0,0.0,0.0}
\definecolor{cv4v1}{rgb}{0.0,0.0,0.0}
\definecolor{cv4v5}{rgb}{0.0,0.0,0.0}
\definecolor{cv4v6}{rgb}{0.0,0.0,0.0}
\definecolor{cv4v7}{rgb}{0.0,0.0,0.0}
\definecolor{cv5v0}{rgb}{0.0,0.0,0.0}
\definecolor{cv5v1}{rgb}{0.0,0.0,0.0}
\definecolor{cv5v6}{rgb}{0.0,0.0,0.0}
\definecolor{cv5v7}{rgb}{0.0,0.0,0.0}
\definecolor{cv6v0}{rgb}{0.0,0.0,0.0}
\definecolor{cv6v2}{rgb}{0.0,0.0,0.0}
\definecolor{cv6v7}{rgb}{0.0,0.0,0.0}
\definecolor{cv7v1}{rgb}{0.0,0.0,0.0}
\definecolor{cv7v2}{rgb}{0.0,0.0,0.0}
\Vertex[style={minimum size=1.0cm,draw=cv0,fill=cfv0,text=clv0,shape=circle},LabelOut=false,L=\hbox{$0$},x=2.5cm,y=5.0cm]{v0}
\Vertex[style={minimum size=1.0cm,draw=cv1,fill=cfv1,text=clv1,shape=circle},LabelOut=false,L=\hbox{$1$},x=0.7322cm,y=4.2678cm]{v1}
\Vertex[style={minimum size=1.0cm,draw=cv2,fill=cfv2,text=clv2,shape=circle},LabelOut=false,L=\hbox{$2$},x=0.0cm,y=2.5cm]{v2}
\Vertex[style={minimum size=1.0cm,draw=cv3,fill=cfv3,text=clv3,shape=circle},LabelOut=false,L=\hbox{$3$},x=0.7322cm,y=0.7322cm]{v3}
\Vertex[style={minimum size=1.0cm,draw=cv4,fill=cfv4,text=clv4,shape=circle},LabelOut=false,L=\hbox{$4$},x=2.5cm,y=0.0cm]{v4}
\Vertex[style={minimum size=1.0cm,draw=cv5,fill=cfv5,text=clv5,shape=circle},LabelOut=false,L=\hbox{$5$},x=4.2678cm,y=0.7322cm]{v5}
\Vertex[style={minimum size=1.0cm,draw=cv6,fill=cfv6,text=clv6,shape=circle},LabelOut=false,L=\hbox{$6$},x=5.0cm,y=2.5cm]{v6}
\Vertex[style={minimum size=1.0cm,draw=cv7,fill=cfv7,text=clv7,shape=circle},LabelOut=false,L=\hbox{$7$},x=4.2678cm,y=4.2678cm]{v7}
\Edge[lw=0.05cm,style={post,color=cv0v3,},](v0)(v3)
\Edge[lw=0.05cm,style={post,color=cv0v4,},](v0)(v4)
\Edge[lw=0.05cm,style={post,color=cv0v7,},](v0)(v7)
\Edge[lw=0.05cm,style={post,color=cv1v3,},](v1)(v3)
\Edge[lw=0.05cm,style={post,color=cv1v6,},](v1)(v6)
\Edge[lw=0.05cm,style={post,color=cv2v0,},](v2)(v0)
\Edge[lw=0.05cm,style={post,color=cv2v4,},](v2)(v4)
\Edge[lw=0.05cm,style={post,color=cv2v5,},](v2)(v5)
\Edge[lw=0.05cm,style={post,color=cv3v4,},](v3)(v4)
\Edge[lw=0.05cm,style={post,color=cv3v5,},](v3)(v5)
\Edge[lw=0.05cm,style={post,color=cv3v6,},](v3)(v6)
\Edge[lw=0.05cm,style={post,color=cv3v7,},](v3)(v7)
\Edge[lw=0.05cm,style={post,color=cv4v1,},](v4)(v1)
\Edge[lw=0.05cm,style={post,color=cv4v5,},](v4)(v5)
\Edge[lw=0.05cm,style={post,color=cv4v6,},](v4)(v6)
\Edge[lw=0.05cm,style={post,color=cv4v7,},](v4)(v7)
\Edge[lw=0.05cm,style={post,color=cv5v0,},](v5)(v0)
\Edge[lw=0.05cm,style={post,color=cv5v1,},](v5)(v1)
\Edge[lw=0.05cm,style={post,color=cv5v6,},](v5)(v6)
\Edge[lw=0.05cm,style={post,color=cv5v7,},](v5)(v7)
\Edge[lw=0.05cm,style={post,color=cv6v0,},](v6)(v0)
\Edge[lw=0.05cm,style={post,color=cv6v2,},](v6)(v2)
\Edge[lw=0.05cm,style={post,color=cv6v7,},](v6)(v7)
\Edge[lw=0.05cm,style={post,color=cv7v1,},](v7)(v1)
\Edge[lw=0.05cm,style={post,color=cv7v2,},](v7)(v2)
\end{tikzpicture}

\end{center}